\documentclass[reqno,11pt,oneside]{amsart}
\usepackage{amsmath,amsthm,amssymb,amsfonts,epsfig,color,graphicx,enumerate,accents,subfigure,mathtools}
\usepackage{hyperref}
\usepackage[a4paper,textwidth = 12cm]{geometry}
\usepackage{cleveref}
\usepackage{todonotes}
\usepackage{microtype}
\usepackage{lineno}
\usepackage[english]{babel}
\usepackage[utf8]{inputenc}
\usepackage[T1]{fontenc}

\newtheorem{thr}{Theorem}[section]
\newtheorem{prop}[thr]{Proposition}
\newtheorem{lemma}[thr]{Lemma}
\newtheorem{cor}[thr]{Corollary}

\theoremstyle{remark}

\theoremstyle{definition}

\theoremstyle{plain}
\newtheorem*{ass*}{\textbf{\textup{(A)}}}

\renewcommand{\det}{\operatorname{det}}

\renewcommand{\P}{\mathcal{P}}
\renewcommand{\d}{\, \text{d}}

\newcommand{\R}{\mathbb{R}}

\newcommand{\spt}{\operatorname{spt}}

\title[Sharp convergence for porous medium stationary states]
{Sharp Convergence Rate for stationary states
of the Porous Medium Equation}

\author{Igor Pinheiro}
\address{Dipartimento di Matematica e Informatica ``Ulisse Dini" \\ Università degli Studi di Firenze}
\email{igorvinicius.pereirapinheiro@unifi.it}

\author{Filippo Santambrogio}
\address{Universit\'e Lyon 1, Ecole Centrale de Lyon, INSA Lyon, Universit\'e Jean Monnet, CNRS, ICJ, UMR 5208, Villeurbanne, France}
\email{santambrogio@math.univ-lyon1.fr}

\date{\today}

\begin{document}

\begin{abstract}
In this paper we prove the sharp rate of convergence for stationary states of the porous medium equation with an external drift in the p-Wasserstein metric under very mild assumptions on the potential and for any $1\leq p\leq +\infty$.
\end{abstract}

\maketitle

\medskip
\textbf{Keywords:} Porous Medium Equation, Wasserstein metric, Optimal Transport, Convergence Rate. 

\section{Introduction}

We consider the porous medium equation with drift \[ \partial_t\rho=\Delta\rho^m+\nabla\cdot(\rho\nabla V), \qquad m>1, \] where $V:\R^d\to\R$ is a prescribed confining potential. This equation describes the evolution of a density under nonlinear diffusion and the drift induced by $-\nabla V$, and is related to models of crowd motion and tissue growth. Under suitable assumptions on $V$, it can be viewed as the gradient flow in the $2$-Wasserstein distance of the energy \[ F_m(\rho)\coloneq\frac{1}{m-1}\int_{\R^d}\rho^m\,\d x+\int_{\R^d}V\rho\,\d x. \] This interpretation belongs to the framework initiated by Jordan, Kinderlehrer and Otto \cite{JKO98} and developed by Ambrosio, Gigli and Savar\'e \cite{AGS05}.

As $m\to+\infty$, the internal energy increasingly penalizes densities larger than one, leading to the constrained functional \[ F_\infty(\rho)\coloneq\begin{cases}
\displaystyle\int_{\R^d}V\rho\,\d x,&\text{if }0\leq\rho\leq1\text{ a.e.},\\
+\infty,&\text{otherwise}.
\end{cases} \] Here the energies are defined on probability densities and extended by $+\infty$ to measures without a density. The gradient flow of $F_\infty$ has been used in modeling congested crowd motion, where the preferred velocity $-\nabla V$ is modified to preserve the constraint $\rho\leq1$ \cite{MRS10,MRSV11}. The passage from the porous medium equation to this constrained evolution is known as the incompressible or Hele--Shaw limit.

The connection between these models was made rigorous by Alexander, Kim and Yao \cite{AKY14}, who established convergence in $W_2$ and its relation to a Hele--Shaw free boundary problem. Quantitative estimates were also studied for related aggregation models in \cite{CKY18} and in the homogeneous $H^{-1}$-norm in \cite{DDP22}. We refer the reader to the introduction of \cite{DMS24} for a detailed account of these developments and further references.

David, M\'esz\'aros and the second author \cite[Theorem 2.4]{DMS24} obtained an improved convergence rate of order $m^{-1/2}$ in $W_2$. For the equation considered here, when $V$ has bounded Hessian and is uniformly convex, their estimate holds globally in time: \[ \sup_{t\geq0}W_2(\rho_m(t),\rho_\infty(t))\leq\frac{C}{\sqrt m}, \] for solutions starting from a common initial datum $\rho_0\in\P_2(\R^d)$ with $0\leq\rho_0\leq1$. The uniformity in time follows from the contractivity of the gradient flows in the $2$-Wasserstein distance.

The same paper also examines stationary states, which are the minimizers of $F_m$ and $F_\infty$. Besides proving an $L^1$ convergence rate of order $m^{-1}$, the authors obtain faster stationary $W_2$ rates under the additional assumption that $\spt(\rho_m)\subseteq\spt(\rho_\infty)$ for sufficiently large $m$ \cite[Theorems 2.6 and 2.7]{DMS24}. Our interest is in determining the sharp convergence rate for these stationary states. 

The following Theorem is our main result.
\begin{thr}\label{thr:main} Let $V \in C^2(\R^d)$ be a coercive function, $V(0) = 0$ and $0$ be the unique critical point of $V$. Define, for $m\geq2$, $\rho_m(x) = \left( \frac{m-1}{m}(C_m - V(x))_+ \right)^{\frac{1}{m-1}}$ and $\rho_\infty(x) = \chi_{\{V\leq C_\infty\}}$ where $C_m$ and $C_\infty$ are chosen so that $\rho_m\d x$ and $\rho_\infty \d x$ are probability measures. Then, for sufficiently large $m$ and any $p\in [1,+\infty]$, the following holds
\begin{equation}\label{eq:main-result}
 \frac{1}{m-1} \lesssim W_p(\rho_m,\rho_\infty ) \lesssim \frac{1}{m-1}.
\end{equation}
\end{thr}

Let us first discuss the implication of our result on the literature. Our upper bound on \eqref{eq:main-result} shows, as it was initially conjectured in \cite{DMS24}, that the rate of convergence for stationary states solutions in $W_2$ can indeed be improved to be of polynomial order $m^{-1}$. This does not, however, guarantee that the rate of convergence for non-stationary can be improved to $m^{-1}$. It could happen that, locally in time, only a slower rate is attainable. What our sharpness statement guarantees is that, globally in time, the rate of convergence in $W_2$ cannot be improved further than $m^{-1}$. Thus, it remains open to establish whether this rate is sharp in the 2-Wasserstein metric for non-stationary state solutions both locally and globally in time.

Let us now discuss the assumptions on the potential $V$. As we discussed before, in \cite{DMS24} the authors examined the same question of convergence under the assumption that $\alpha I \leq D^2 V \leq AI$ for some $0 <\alpha \leq A$ where $I$ denotes the identity matrix. Our assumptions that $V \in C^2(\R^d)$, $V$ coercive and $V$ has a unique critical point are weaker than theirs. Our extra assumption that $0$ is such unique global minimizer is chosen cosmetically and of course nothing changes if the minimum point or the minimum value are different. A fine analysis of the proof shows that it would work with almost no modifications also in the case where $0$ is the only {\it critical level} (allowing thus for many global minimizers), provided the Lebesgue measure of such a level (if postive) is strictly smaller than $1$. Actually the only important point is that $C_\infty$ itself is not a critical level, but we prefer to express our assumptions differently, since $C_\infty$ is not necessarily explicitly known. 
Also, the global regularity assumptions we impose on $V$ (namely, $V$ of class $C^2$ on $\R^d$ and with a unique critical point) need not to hold on the entire domain $\R^d$, but it would suffice to assume them only in a bounded region containing the relevant sublevel sets of $V$. This would allow for potentials $V$ with much wilder behavior. However, we do not see any real gains in going to this level of generality, except making the language more cumbersome. Similarly, the coercivity assumption on $V$ is what guarantees that the sublevel sets of $V$, and hence the support of $\rho_m$ and $\rho_\infty$, are compact sets. This could again be weakened to assuming only that the sublevel sets of $V$ in a neighborhood of $C_\infty$ are compact. Of course, one would need to assume separately then that $V$ admits a unique global minimizer. 

Our proof strategy consists in first building a flow $\Psi_m$ that moves $\spt(\rho_\infty)$ to $\spt(\rho_m)$ with unit speed and, among other things, satisfies $\Vert \Psi_m - \operatorname{Id} \Vert_{L^\infty} \lesssim |C_m - C_\infty|$ (\Cref{cor:flow-application}). Then, defining $\tilde \rho_m = (\Psi_m^{-1})_{\#} \rho_m$, we use the Dacorogna--Moser construction to obtain the estimate $W_p(\tilde \rho_m, \rho_\infty) \lesssim \frac{1}{m-1} + |C_m - C_\infty|$ for $1\leq p\leq m/2$, with an implicit constant independent of $m$ and $p$ (\Cref{prop:fixed-support}). By establishing that $|C_m-C_\infty|\lesssim\frac{1}{m-1}$ (\Cref{lem:support-estimate}) and applying the estimate of Jylh\"a and Rajala \cite{RJ16} (which is a refined version of the result by Bouchitt\'e, Jimenez and Rajesh \cite{BJR07}) with $p=m/2$, we obtain $W_\infty(\tilde \rho_m,\rho_\infty)\lesssim\frac{1}{m-1}$ (\Cref{lem:W-infty-rate}). Moreover, we also show that, due to the good properties of $\Psi_m$, $W_\infty(\rho_m, \tilde \rho_m) \lesssim |C_m-C_\infty|$. Finally, we prove the upper bound by the triangle inequality for the metric $W_\infty$ and the monotonicity of the Wasserstein distances. As for the lower bound, we show that $W_1(\rho_m,\rho_\infty)\gtrsim\frac{1}{m-1}$, which gives the lower bound for every $1\leq p\leq+\infty$ by monotonicity.
\section{Preliminary geometric results}

In this section we establish the results that will be used repeated times in the proofs to come. We also establish construction of the flow that will perform the geometric modification of $\rho_m$ to a measure on the support of $\rho_\infty$ in \Cref{cor:flow-application}. This will be the key geometrical step that will allow us to employ the Dacorogna-Moser construction in the next section.

The proofs of the next two results do not contribute to the clarity of the main argument, therefore we postpone them to the \Cref{app:geometric-results}.

\begin{prop}\label{prop:flow} Let $V$ be as in \Cref{thr:main}. For every $C>0$, there exist $\delta>0$, a compactly supported vector field $b \in C_c^{1}(\R^d ; \R^d)$ and a constand $K= K(V,C)$ such that the flow generated by $b$, namely
\begin{equation}\label{eq:flow}
  \begin{cases}
    \partial_s\Psi_s(x) = b(\Psi_s(x)) \\
    \Psi_0(x) = x
  \end{cases}
\end{equation}
is a $C^1(\R^d)$ diffeomorphism for every $s\in \R$, with inverse $\Psi_{-s}$. Moreover, for every $|s| \leq \delta$, we have $\Psi_s(\{V \leq C\}) = \{V\leq C+s\}$, and the following estimate holds:
\begin{equation}\label{eq:flow-estimate}
  \Vert \Psi_s - \operatorname{id} \Vert_{L^\infty(\R^d)} + \Vert \operatorname{det} D\Psi_s - 1 \Vert_{L^\infty(\R^d)} \leq K|s|.
\end{equation} 
\end{prop}

\begin{prop}\label{prop:area-formula} Consider $V$ as in \Cref{thr:main}. Define $A(t) = |\{V \leq t\}|$, then $A\in C^0([0,+\infty))\cap C^1((0,+\infty))$. For $ t \in (0,+\infty)$ we have
    \[A'(t) = \int_{V=t} \frac{1}{|\nabla V|} \d \mathcal H^{d-1}.\]
Moreover, $A$ is strictly increasing on $[0,+\infty)$.
\end{prop}

 \begin{lemma}\label{lem:cm-convergence}
   Let $V$, $\rho_m$ and $\rho_\infty$ be as in \Cref{thr:main}. Then, $C_m \to C_\infty.$
\end{lemma}

\begin{proof} It suffices to show that $A(C_m)\to A(C_\infty)=1$, as $t\mapsto A(t)$ is strictly increasing and continuous by \Cref{prop:area-formula}. Choose constants $C_0$ and $C_1$ such that $0<C_0<C_\infty<C_1$ and define $I_m:\R\to\R_+$ by
\[
I_m(C)=\int_{\R^d}\left(\frac{m-1}{m}(C-V)_+\right)^{\frac{1}{m-1}}\d x.
\]
For ease of notation, let $K_m\coloneq\left(\frac{m-1}{m}\right)^{\frac{1}{m-1}}$. Since $V$ is coercive, it admits a global minimizer. Any global minimizer is a critical point of $V$, and since $0$ is the unique critical point by assumption, it follows that $0$ is the unique minimizer of $V$. As $V(0)=0$, we have $V(x)>0$ for every $x\neq 0$. By the coarea formula and \Cref{prop:area-formula}, we have
\begin{equation}\label{eq:coarea-im}\begin{aligned}
    I_m(C)&=K_m\int_0^C(C-s)^{\frac{1}{m-1}}\left(\int_{\{V=s\}}\frac{1}{|\nabla V|}\d\mathcal H^{d-1}(x)\right)\d s\\ &=K_m\int_0^C(C-s)^{\frac{1}{m-1}}A'(s)\d s.\end{aligned}
\end{equation}
Now we estimate, for $C\in[C_0,C_1]$,
\begin{equation}\label{eq:i-area-estimate}
\begin{aligned}
|I_m(C)-&A(C)|=\left|K_m\int_0^C(C-s)^{\frac{1}{m-1}}A'(s)\d s-\int_0^CA'(s)\d s\right|\\
&\leq K_m\int_0^C\left|(C-s)^{\frac{1}{m-1}}-1\right|A'(s)\d s+|K_m-1|A(C)\\
&\leq \int_0^C\left|(C-s)^{\frac{1}{m-1}}-1\right|A'(s)\d s+|K_m-1|A(C)\\
&\lesssim\frac{1}{m-1}\int_0^C|\ln(C-s)|A'(s)\d s+|K_m-1|A(C).
\end{aligned}
\end{equation}
In passing to the third line we used $0<K_m\leq 1$. For the last inequality, for every $0\leq s<C$, the mean value theorem applied to $r\mapsto (C-s)^r$ gives some $\Theta=\Theta(s,m)\in\left(0,\frac{1}{m-1}\right)$ such that
\[\left|(C-s)^{\frac{1}{m-1}}-1\right|
\leq\frac{1}{m-1}|\ln(C-s)|(C-s)^\Theta.\]
Since $m\geq 2$, we have $0\leq \Theta\leq 1$, and since $0<C-s\leq C_1$, it follows that $(C-s)^\Theta\leq \max\{1,C_1\}.$
Hence, the implicit constant in \eqref{eq:i-area-estimate} is independent of $m$. The logarithmic integral in \eqref{eq:i-area-estimate} is bounded uniformly for $C\in[C_0,C_1]$. Indeed, the singularity at $s=C$ remains uniformly separated from $s=0$, and $A'$ is bounded on $[C_0/2,C_1]$. The claim therefore follows from $A'\in L^1(0,C_1)$ and the integrability of $|\ln u|$ near zero.

Therefore,
\begin{equation}\label{eq:uniform-log-estimate}
    \sup_{C\in[C_0,C_1]}\int_0^C|\ln(C-s)|A'(s)\d s<+\infty.
\end{equation}

It remains to estimate $K_m-1$. We have
\begin{equation}\label{eq:km-estimate}
|K_m-1|\leq \frac{1}{m-1}\left|\ln\left(\frac{m-1}{m}\right)\right|\leq\frac{1}{(m-1)^2}.
\end{equation}
Since $A(C)\leq A(C_1)$ for every $C\in[C_0,C_1]$, it follows, by substituting \eqref{eq:km-estimate} and \eqref{eq:uniform-log-estimate} into \eqref{eq:i-area-estimate} that
\begin{equation}\label{eq:uniform-im-area}
    \sup_{C\in[C_0,C_1]}|I_m(C)-A(C)|\lesssim\frac{1}{m-1}.
\end{equation}

We argue that $C_m\in[C_0,C_1]$ for sufficiently large $m$. Since $C_\infty\in(C_0,C_1)$ and $A$ is strictly increasing, we have $A(C_0)<1<A(C_1)$. By \eqref{eq:uniform-im-area}, for sufficiently large $m$ we also have $I_m(C_0)<1<I_m(C_1)$. Since $I_m$ is continuous and strictly increasing on $[C_0,C_1]$
 and, by the normalization of $\rho_m$, $I_m(C_m)=1$, it follows that $C_m\in[C_0,C_1]$ for sufficiently large $m$. Substituting $C=C_m$ into \eqref{eq:uniform-im-area}, we obtain
\begin{equation}\label{eq:area-cm-estimate}
    |A(C_m)-1|=|A(C_m)-I_m(C_m)|\lesssim\frac{1}{m-1}.
\end{equation}
Since $1=A(C_\infty)$, we conclude that $A(C_m)\to A(C_\infty)$ and, since $A$ is continuous and strictly increasing, this implies that $C_m\to C_\infty$.
\end{proof}

\begin{cor}\label{cor:flow-application} Let $V$, $\rho_m$ and $\rho_\infty$ be as in \Cref{thr:main}. For $m$ sufficiently large, there exists a diffeomorphism $\Psi_m: \R^d \to \R^d$ and constants $\delta,K>0$ such that $\Psi_m(\{V\leq C_\infty\}) = \{V\leq C_m\}$\begin{equation}\label{eq:psi-m-estimate}
    \Vert \Psi_m  - \operatorname{Id} \Vert_{L^\infty(\{V\leq C_\infty\})} + \Vert \operatorname{det} D \Psi_m - 1 \Vert_{L^\infty(\{V\leq C_\infty\})} \leq K|C_m - C_\infty|.
\end{equation}And for $x\in \{C_\infty - \frac{\delta}{3} \leq V \leq C_\infty + \frac{\delta}{3}\}$, we have \begin{equation}\label{eq:collar-identity}C_m - V(\Psi_m(x)) = C_\infty - V(x)\end{equation}
\end{cor}

\begin{proof} Let $C_\infty > 0$ play the role of $C$ in \Cref{prop:flow}, then there exists $\delta >0$ and a $C^1(\R^d)$ diffeomorphism, $\Psi_s$, for $|s|\leq \delta$ we have $\Psi_s(\{V\leq C_\infty\}) = \{V\leq C_\infty+s\}$. Choosing $s = C_m - C_\infty$ and denoting $\Psi_m \coloneq \Psi_{C_m - C_\infty}$ we obtain $\Psi_m(\{V \leq C_\infty\}) = \{V\leq C_m\}$. Moreover, for $m$ sufficiently large we have that $|C_m - C_\infty| < \delta$, since $C_m \to C_\infty$ by \Cref{lem:cm-convergence}. Thus, by \eqref{eq:flow-estimate}, estimate \eqref{eq:psi-m-estimate} holds. 

It remains to prove \eqref{eq:collar-identity}. Recall that $\chi=1$ on $\{C_\infty-\delta\leq V\leq C_\infty+\delta\}$ (see \Cref{app:geometric-results}). Since $C_m\to C_\infty$, for $m$ sufficiently large we have $|C_m-C_\infty|\leq\delta/3$. Let $s_m\coloneq C_m-C_\infty$ and fix $x\in\{C_\infty-\delta/3\leq V\leq C_\infty+\delta/3\}$. We claim that the trajectory $t\mapsto\Psi_{ts_m}(x)$ remains in the region where $\chi=1$ for every $t\in[0,1]$. Indeed, otherwise let $t_0$ be its first exit time from $\{C_\infty-\delta<V<C_\infty+\delta\}$. For every $t<t_0$ we have $\chi(\Psi_{ts_m}(x))=1$, and therefore
\begin{equation}\label{eq:pre-exit-identity}
\frac{\d}{\d t}V(\Psi_{ts_m}(x))=s_m\nabla V(\Psi_{ts_m}(x))\cdot b(\Psi_{ts_m}(x))=s_m.
\end{equation}
Hence $V(\Psi_{ts_m}(x))=V(x)+ts_m$ for every $t<t_0$, and consequently
\[
|V(\Psi_{ts_m}(x))-C_\infty|\leq |V(x)-C_\infty|+t|s_m|\leq\frac{2\delta}{3}<\delta,
\]
which contradicts the definition of $t_0$. Thus the trajectory remains entirely in the region where $\chi=1$, and integrating \eqref{eq:pre-exit-identity} on $[0,1]$ and recalling $s_m = C_m - C_\infty$ gives
\[
V(\Psi_m(x))=V(x)+C_m-C_\infty.
\]
Rearranging, we obtain
\[
C_m-V(\Psi_m(x))=C_\infty-V(x)
\]
for every $x\in\{C_\infty-\delta/3\leq V\leq C_\infty+\delta/3\}$.
    
\end{proof}

\section{The case of fixed support.}

In this section we establish the estimate $W_p((\Psi_m)^{-1}_\#\rho_m, \rho_\infty) \lesssim \frac{1}{m-1} + |C_m - C_\infty|$ for every $1\leq p\leq m/2$. As both  $(\Psi_m)^{-1}_\#\rho_m$ and $\rho_\infty$ have support $\{V\leq C_\infty\}$, we may construct a map transporting one to the other via the Dacorogna-Moser construction. Moreover, we can estimate from above the quadratic cost of such a map by $\Vert (\Psi_m)^{-1}_\#\rho_m -   \rho_\infty \Vert_q$ with a constant depending only on the support, and as the support is fixed, such a constant will not depend on $m$. Finally, we estimate $\Vert (\Psi_m)^{-1}_\#\rho_m -   \rho_\infty \Vert_q$. This is the content of the next Proposition.

\begin{prop}\label{prop:l-estimate} Let $V$, $\rho_m$ and $\rho_\infty$ be as in \Cref{thr:main}. Let $\tilde \rho_m = (\Psi_m)^{-1}_{\#} \rho_m$ where $\Psi_m$ is the diffeomorphism from \Cref{cor:flow-application}, then, for $m$ sufficiently large and $1\leq q <+\infty$, the following holds
\begin{equation}
    \Vert \tilde \rho_m - 1\Vert_{L^{q}(\{V\leq C_\infty\})} \lesssim \frac{1}{m-1} + |C_m - C_\infty| 
\end{equation}
\end{prop}

\begin{proof} For ease of notation let $K_m \coloneq ((m-1)/m)^{\frac{1}{m-1}}$, and rewrite the density $\rho_m$  as $\rho_m(x) = K_m(C_m - V(x))_+^{\frac{1}{m-1}}$. The relation $\tilde \rho_m(x) \d x = (\Psi_m^{-1})_{\#} \rho_m \d x$, which is equivalent to
\begin{equation}
    \tilde \rho_m(x) = \rho_m(\Psi_m(x)) \det(D\Psi_m(x)).
\end{equation} By the equation above, we have
\begin{equation}
\begin{aligned}
|&\tilde\rho_m(x)-1|^q
=
\left|
    K_m \det D\Psi_m(x)
    \bigl(C_m-V(\Psi_m(x))\bigr)^{\frac{1}{m-1}}
    -1
\right|^q
\\
&\le
\left(
\begin{aligned}
&K_m \left|\det D\Psi_m(x)\right|
 \left|
    \bigl(C_m-V(\Psi_m(x))\bigr)^{\frac{1}{m-1}}
    -1
 \right|
\\
&\qquad
 +\left|K_m\det D\Psi_m(x)-1\right|
\end{aligned}
\right)^q
\\
&\le
2^{q}
\left(
\begin{aligned}
&(K_m)^q
 \left|\det D\Psi_m(x)\right|^q
 \left|
    \bigl(C_m-V(\Psi_m(x))\bigr)^{\frac{1}{m-1}}
    -1
 \right|^q
\\
&\qquad
 +\left|K_m\det D\Psi_m(x)-1\right|^q
\end{aligned}
\right).
\end{aligned}
\end{equation} and thus
\begin{equation}\label{eq:lp-main}
\begin{aligned}
\|\tilde\rho_m-1\|_{L^q(\Omega_\infty)}^q
&\leq
2^{q}
\Bigg[
(K_m)^q
\|\det D\Psi_m\|_{L^\infty(\Omega_\infty)}^q
\\
&\qquad\quad {}\times
\left\|
    \bigl(C_m-V\circ\Psi_m\bigr)^{\frac{1}{m-1}}-1
\right\|_{L^q(\Omega_\infty)}^q
\\
&\qquad {}
+\left\|
    K_m\det D\Psi_m-1
\right\|_{L^q(\Omega_\infty)}^q
\Bigg].
\end{aligned}
\end{equation}
We will estimate each of the terms and factors in \eqref{eq:lp-main}. Let \[F_1 = (K_m)^q, F_2 = \Vert \det D \Psi_m(x)\Vert^q_{L^\infty}\]and \[T_1 = \Vert (C_m - V(\Psi_m(x)))^{\frac{1}{m-1}} -1 \Vert^q_{L^q(\{V\leq C_\infty\})},\]\[ T_2 =  \Vert K_m \det D \Psi_m(x) - 1 \Vert^q_{L^q(\{V\leq C_\infty\})}.\]

\vspace{0.3cm}

We start with $T_1$:
\begin{equation}\label{eq:t1-reduction}\begin{aligned}
    \Vert (C_m - &V(\Psi_m(x)))^{\frac{1}{m-1}} -1  \Vert^q_{L^q(\{V\leq C_\infty\})}\\ &= \int_{\{V\leq C_\infty\}} \left | (C_m - V(\Psi_m(x)|^{\frac{1}{m-1}} -1 \right |^q \d x \\
    &= \int_{\{V\leq C_\infty\}} \left | \exp\left(\frac{1}{m-1} \ln(C_m - V(\Psi_m(x))\right) -1 \right |^q \d x \\
    & \lesssim \frac{1}{(m-1)^q} \int_{\{V\leq C_\infty\}} |\ln(C_m - V(\Psi_m(x))|^q \d x   \end{aligned}
\end{equation} where, in the last inequality, we used that the exponent $\frac{1}{m-1}\ln(C_m-V(\Psi_m(x)))$ is bounded above uniformly in $m$ and $x$, since $C_m$ is uniformly bounded, and that the exponential function is Lipschitz on any interval bounded above.

We must show that \begin{equation}\label{eq:t1-log-integrability}
    \sup_{m} \int_{\{V\leq C_\infty\}} |\ln(C_m - V(\Psi_m(x))|^q \d x < +\infty.
\end{equation}Let $\delta>0$ be as in \Cref{cor:flow-application}. We may split the integral above into the regions $\{V< C_\infty -\frac{\delta}{3} \}$ and $\{C_\infty - \frac{\delta}{3} \leq V\leq C_\infty\}$.

\begin{equation}\begin{aligned}
    \Vert (C_m - &V(\Psi_m(x)))^{\frac{1}{m-1}} -1 \Vert^q_{L^q(\{V\leq C_\infty\})} \\&\lesssim \frac{1}{(m-1)^q} \bigg(  \underbrace{\int_{\{V\leq C_\infty - \frac{\delta}{3}\}} |\ln(C_m - V(\Psi_m(x))|^q \d x}_{T_{1.1}} \\& +\underbrace{\int_{\{C_\infty - \frac{\delta}{3} \leq V\leq C_\infty\}} |\ln(C_m - V(\Psi_m(x))|^q \d x}_{T_{1.2}}\bigg)
\end{aligned}
\end{equation}
By \eqref{eq:collar-identity} and the co-area formula, if $f(s) = \int_{V=s} |\nabla V|^{-1} \d \mathcal{H}^{d-1}$, then
\begin{equation}
    \begin{aligned}
        T_{1.2} = \int_{\{C_\infty -\frac{\delta}{3}\leq  V \leq C_\infty\}} &|\ln(C_m - V(\Psi_m(x))|^q \d x \\&= \int_{\{C_\infty - \frac{\delta}{3} \leq V\leq C_\infty \}} |\ln(C_\infty - V(x))|^q \d x \\
        &= \int^{C_\infty}_{C_\infty-\frac{\delta}{3} }|\ln(C_\infty - s)|^q f(s) \d s
    \end{aligned}
\end{equation}
fix $\varepsilon >0$ small, splitting away from the singularity we have
\begin{equation}
    \begin{aligned}\label{eq:boundary-log-split}
         \int^{C_\infty - \varepsilon}_{C_\infty-\frac{\delta}{3} }|\ln(C_\infty - s)|^q f(s) \d s +\int^{C_\infty }_{C_\infty - \varepsilon }|\ln(C_\infty - s)|^q f(s) \d s <+\infty
    \end{aligned}
\end{equation}the first integral is finite because $s\mapsto |\ln(C_\infty - s)|^q$ is bounded from above and below on $[C_\infty - \frac{\delta}{3} , C_\infty - \varepsilon]$ and $s\mapsto f(s)$ is integrable on the same interval (see \Cref{prop:area-formula}) also the second integral is finite since $s\mapsto f(s)$ is bounded on $[C_\infty - \varepsilon , C_\infty]$ while $s\mapsto |\ln(C_\infty - s)|^q$ is integrable on the same region. As neither integral in \eqref{eq:boundary-log-split} depends on $m$, $T_{1.2}$ is bounded uniformly in $m$. 

From \Cref{lem:cm-convergence} we have $C_m\to C_\infty$, also from \Cref{cor:flow-application} we have that $\Psi_m \to \operatorname{Id}$ uniformly. Since $V$ is continuous and coercive, the set $\{V\leq C_\infty - \frac{\delta}{3}\}$ is compact and $V$ is uniformly continuous there. We obtain that $C_m - V(\Psi_m(x)) \to C_\infty - V(x)$ uniformly in $x$. As on $\{V\leq C_\infty - \frac{\delta}{3}\}$ we have $C_\infty - V(x) \geq \frac{\delta}{3} > 0$, for $m$ sufficiently large we also have $C_m - V(\Psi_m(x)) \geq \frac{\delta}{3} > 0$, thus $|\ln(C_m - V(\Psi_m(x)))|$ is bounded from above and below uniformly in $x$, and the bounds can be chosen independently of $m$, hence $T_{1.1}$ is bounded uniformly in $m$.

Therefore, from \eqref{eq:t1-reduction} we have  \begin{equation}\label{eq:t1-estimate}
    T_1 = \Vert (C_m - V(\Psi_m(x)))^{\frac{1}{m-1}} -1 \Vert^q_{L^q(\{V\leq C_\infty\})} \lesssim \frac{1}{(m-1)^q}.
\end{equation}

\vspace{0.3cm}

We have already estimated $F_1$ in \eqref{eq:km-estimate} to be 
\begin{equation}\label{eq:km-rate}
    \begin{aligned}
        |K_m - 1| \lesssim \frac{1}{(m-1)^2}.
    \end{aligned}
\end{equation}In particular, $K_m$, and hence $F_1$, is bounded.

\vspace{0.3cm}

By \eqref{eq:psi-m-estimate}, we have $\det D\Psi_m=1+O(|C_m-C_\infty|)$ uniformly on $\{V\le C_\infty\}$. Since $C_m\to C_\infty$ by Lemma~\ref{lem:cm-convergence}, it follows that $F_2$ is uniformly bounded in $m$. Moreover, \eqref{eq:km-rate} gives $K_m=1+O(m^{-2})$, and hence, using $|\{V\le C_\infty\}|=1$,
\begin{equation}\label{eq:t2-estimate}
T_2
=\Vert K_m\det D\Psi_m-1\Vert^q_{L^q(\{V\le C_\infty\})}
\lesssim |C_m-C_\infty|^q+m^{-2q}.
\end{equation}

\vspace{0.3cm}

Returning to \eqref{eq:lp-main} we have, since $(K_m)^q$ and $\Vert \det D\Psi_m(x)\Vert^q_{L^\infty(\{V\leq C_\infty\})}$ are bounded uniformly in $m$, \eqref{eq:t1-estimate} and \eqref{eq:t2-estimate}
\[\Vert \tilde \rho_m - 1 \Vert^q_{L^q(\{V\leq C_\infty\})} \lesssim |C_m - C_\infty|^q + \frac{1}{(m-1)^q}\] and thus
\[\Vert \tilde \rho_m - 1 \Vert_{L^q(\{V\leq C_\infty\})} \lesssim |C_m - C_\infty| + \frac{1}{(m-1)}.\]

\end{proof} 

Define $u_m$ as a solution to the following Neumann problem
\begin{equation}\label{eq:neumann-problem}
    \begin{cases}
        -\Delta u_m  = \tilde \rho_m - \rho_\infty \text{ in } \{V<C_\infty\} \\
        \partial_n u_m = 0 \text{ on } \{V=C_\infty\}.
    \end{cases}
\end{equation} As $\tilde \rho_m\d x$ and $\rho_\infty\d x$ are probability measures we have $\int (\tilde \rho_m -\rho_\infty)\d x = 0$, which guarantees a solution exists (see \cite[Chapter 6, Exercise 3]{Eva22}).

\begin{lemma}\label{lem:neumann-estimate} For $q>d$, a solution $u_m$ of \eqref{eq:neumann-problem}  satisfies the estimate \begin{equation}
    \Vert \nabla u_m  \Vert_{L^\infty(\{V\leq C_\infty\})]} \lesssim \Vert \tilde \rho _m - 1 \Vert_{L^q(\{V\leq C_\infty\})}
\end{equation}
\end{lemma}

\begin{proof}
The set $\{V<C_\infty\}$ is a bounded connected domain with $C^2$ boundary. By elliptic regularity for the Neumann problem (see \cite[Theorem~B.2]{BS25}) and the Sobolev embedding for $q>d$, we obtain
\[
\|\nabla u_m\|_{L^\infty(\{V<C_\infty\})}
\lesssim \|\nabla u_m\|_{W^{1,q}(\{V<C_\infty\})}
\lesssim \|\tilde\rho_m-1\|_{L^q(\{V<C_\infty\})}.
\]
The implicit constants depend only on the domain and $q$, and are therefore independent of $m$.
\end{proof}

\begin{prop}\label{prop:fixed-support} Let $V$, $\rho_m$, $\tilde \rho_m$ and $\rho_\infty$ be as in \Cref{prop:l-estimate}. Then, for sufficiently large $m$ and every $1\leq p \leq m/2 $\[W_p(\tilde \rho_m,\rho_\infty) \lesssim \frac{1}{m-1} + |C_m - C_\infty|.\]with the implicit constant independent of $m$ and $p$.
\end{prop}

 \begin{proof}
It suffices to consider $3\leq p\leq m/2$, since the case $1\leq p<3$ follows by taking $p=3$ and using the monotonicity of the Wasserstein distances.
We do the Dacorogna--Moser interpolation (see \cite[Box 4.3]{San15}) with the flux $w_m:=\nabla u_m$. Set $\rho_t:=(1-t)\rho_\infty+t\tilde \rho_m$ for $t\in[0,1]$. On $\{V\leq C_\infty\}$, this is simply $\rho_t=(1-t)+t\tilde \rho_m$, and the equation $-\Delta u_m=\tilde \rho_m-1$ gives $\partial_t\rho_t+\operatorname{div}w_m=0$, with no boundary flux because $\partial_n u_m=0$ on $\{V= C_\infty \}$. Therefore $(\rho_t,w_m)$ is an admissible pair in the Benamou--Brenier formula, and hence \[ W_p^p(\tilde\rho_m,\rho_\infty) \leq \int_0^1\int_{\{V<C_\infty\}}\frac{|w_m|^p}{\rho_t^{p-1}}\,\d x\,\d t. \] Swapping the order of integration and using that $w_m$ is independent of time gives
\begin{equation}
W_p^p(\tilde\rho_m,\rho_\infty) \leq
\int_{\{V<C_\infty\}} |w_m|^p\int_0^1\frac{1}{\rho_t^{p-1}}\d t\,\d x.
\end{equation}By \Cref{cor:flow-application}, $\Psi_m$ maps the interior of $\{V\leq C_\infty\}$ diffeomorphically into the interior of $\{V\leq C_m\}$, thus $\tilde \rho_m > 0$ on $\{V<C_\infty\}$ and we may compute the integral in time above to obtain
\begin{equation}\label{eq:dm-w2-estimate}\begin{aligned}
W_p^p(\tilde\rho_m,\rho_\infty) &\leq \int_{\{V<C_\infty\}} |w_m|^p\frac{\tilde \rho_m(x)^{2-p}-1}{(p-2)(1-\tilde \rho_m(x))}\d x \\&\leq
\Vert w_m \Vert^p_\infty J_m
\end{aligned}
\end{equation}

where \[ J_m:=\int_{\{V<C_\infty\}}\frac{\tilde \rho_m^{2-p}-1}{(p-2)(1-\tilde \rho_m)}\,\d x. \]
It remains to prove that $J_m^{1/p}$ is bounded uniformly in $m$ and $p$. Since $p\geq3$, \[ 0<\frac{r^{2-p}-1}{(p-2)(1-r)} \leq 1+r^{2-p}, \qquad r>0, \] and $|\{V<C_\infty\}|=1$, we have \[ J_m\leq 1+\int_{\{V<C_\infty\}}\tilde \rho_m^{2-p}\,\d x. \] Recall that $\tilde \rho_m(x) = K_m(C_m - V(\Psi_m(x)))^{\frac{1}{m-1}} \det D \Psi_m(x)$ where we have denoted $K_m \coloneq (\frac{m-1}{m})^{\frac{1}{m-1}}$. Since $K_m$ and $\det D \Psi_m$ are uniformly bounded away from zero for sufficiently large $m$, there exists $0<c\leq1$, independent of $m$, such that $K_m\det D\Psi_m\geq c$ on $\{V<C_\infty\}$. Therefore, \[ \int_{\{V<C_\infty\}}\tilde \rho_m^{2-p}\,\d x \leq c^{2-p}\int_{\{V<C_\infty\}}(C_m-V(\Psi_m(x)))^{\frac{2-p}{m-1}}\,\d x. \] Let $\delta>0$ be as in \Cref{cor:flow-application}. As shown in the proof of \Cref{prop:l-estimate}, $C_m-V(\Psi_m(x))$ is bounded above and below by positive constants uniformly in $m$ on $\{V\leq C_\infty-\frac{\delta}{3}\}$. On the remaining region, \eqref{eq:collar-identity} and the coarea formula show that the integral is bounded uniformly in $m$ and $p$, since $A'$ is bounded near $C_\infty$ and \[ -\frac12<\frac{2-p}{m-1}<0 \] for $3\leq p\leq m/2$. Therefore, there exists $C>0$, independent of $m$ and $p$, such that \[ J_m\leq1+C c^{2-p}\leq(1+C)c^{-p}, \] and hence \[ J_m^{1/p}\leq c^{-1}(1+C)^{1/p}\leq c^{-1}(1+C)^{1/3}. \] By \eqref{eq:dm-w2-estimate}, \Cref{lem:neumann-estimate} and \Cref{prop:l-estimate}, both with exponent $d+1$, we obtain \[ W_p(\tilde \rho_m,\rho_\infty)\leq\Vert w_m\Vert_\infty J_m^{1/p} \lesssim \frac{1}{m-1} + |C_m - C_\infty|. \]

\end{proof}

\section{Main Result}

\begin{lemma}\label{lem:support-estimate} Let $V$, $\rho_m$ and $\rho_\infty$ be as in \Cref{thr:main}. Then, for $m$ sufficiently large,
    \begin{equation}
        |C_m - C_\infty| \lesssim \frac{1}{m-1}.
    \end{equation}
\end{lemma}

\begin{proof} By \Cref{prop:area-formula} we have $A \in C^1((0,+\infty))$ and from $A'(C_\infty) = \int_{V=C_\infty} \frac{1}{|\nabla V|} \d \mathcal H^{d-1} > 0$, there exists constants $c>0$ and $0<\eta <C_\infty$ such that \[A'(t) \geq c >0 \text{ for } [C_\infty - \eta ,C_\infty + \eta].\]

From \Cref{lem:cm-convergence} we have that, for sufficiently large $m$, $C_m \in [C_\infty - \eta , C_\infty + \eta]$. By the mean value Theorem there exists $\xi_m \in  [C_\infty -\eta,C_\infty + \eta]$such that \begin{equation}\label{eq:support-mvt}|A(C_m) - A(C_\infty)|  = A'(\xi_m) |C_m - C_\infty| \geq c |C_m - C_\infty|. \end{equation}

In the proof of \Cref{lem:cm-convergence} we already obtained that $|A(C_m) - A(C_\infty)| \lesssim \frac{1}{(m-1)}$, thus by \eqref{eq:support-mvt} we conclude \[|C_m - C_\infty| \lesssim \frac{1}{m-1}.\]
\end{proof}

For $\mu,\nu \in \P(\R^d)$, the $\infty$-Wasserstein metric (see \cite[Section 3.2]{San15} is defined by 
\[W_\infty (\mu,\nu) \coloneq \inf_{\lambda \in \Pi(\mu,\nu)}\| |x-y| \|_{L^\infty(\lambda)}.\] We recall the following consequence of \cite[Theorem 1.1]{RJ16} for p-power costs.

\begin{thr}[Theorem 1.1 in \cite{RJ16}]\label{thr:l-infty-ineq} Let $\mu\in \P(\R^d)$ have compact and connected support. Define $m(t) \coloneq \inf_{x\in \spt(\mu) } \mu(B(x,t))$. Then, for every $1\leq p <+\infty$
\begin{equation}
    W_p^p(\mu,\nu) \geq  \frac{1}{2}{m\left( \frac{W_\infty(\mu,\nu)}{17} \right)}\left(\frac{W_\infty(\mu,\nu)}{17}\right)^p, 
\end{equation} for every $\nu \in \P(\spt(\mu))$
    
\end{thr}

\begin{lemma}\label{lem:W-infty-rate} Let $V$, $\tilde \rho_m$ and $\rho_\infty$ be as in \Cref{prop:l-estimate}. Then, for sufficiently large $m$, 
\begin{equation}
    W_\infty(\rho_\infty,\tilde \rho_m) \lesssim \frac{1}{m-1}
\end{equation}
\end{lemma}

\begin{proof}
    Since $\{V < C_\infty\}$ is a bounded connected domain with $C^2$ boundary, there exists $c>0$ such that $m(t) \geq ct^d$ for every $0<t< \operatorname{diam}(\{V<C_\infty\})$. Since $\tilde \rho_m \in \P(\spt(\rho_\infty))$, \Cref{thr:l-infty-ineq} gives us that, for every $1\leq p <+\infty$ \begin{equation}\label{eq:W-infty-rhom}
        W_\infty(\rho_\infty,\tilde\rho_m) \leq 17 \left(\frac{2}{c} \right)^{\frac{1}{p+d}} W^{\frac{p}{p+d}}_p(\rho_\infty,\tilde \rho_m).
    \end{equation}The factor $17\left( \frac{2}{c}\right)^{\frac{1}{p+d}}$ is uniformly bounded in $p$, therefore the inequality above can be written with a constant $C>0$ depending only on $d$ and the domain $\{V<C_\infty\}$. 

    Choose $p = m/2$
    , by \Cref{prop:fixed-support} and \Cref{lem:support-estimate} and \Cref{eq:W-infty-rhom}
    \[W_\infty(\rho_\infty,\tilde \rho_m) \lesssim \frac{1}{(m-1)^{\frac{m}{m+2d}}} \lesssim \frac{1}{m-1}\]since $(m-1)^{\frac{2d}{m+2d}}\to 1$ as $m\to\infty$, and thus it is uniformly bounded in $m$.
\end{proof}

Finally, we are ready to prove the main theorem.

\begin{proof}[Proof of \Cref{thr:main}]
    Let $\tilde \rho_m = (\Psi^{-1}_{m})_\# \rho_m$ where $\Psi_{m}$ is the diffeomorphism given by \cref{cor:flow-application}. The triangular inequality gives
    \begin{equation}\label{eq:w2-triangle}
        W_\infty(\rho_m,\rho_\infty) \leq W_\infty(\rho_m,\tilde \rho_m) + W_\infty(\tilde \rho_m, \rho_\infty).
    \end{equation}
Observe that $\Psi_{m}^{-1}$ is a candidate transport map from $\rho_m$ to $\tilde \rho_m$, therefore, by \Cref{cor:flow-application} equation \eqref{eq:flow-estimate} we have 
    \begin{equation}\label{eq:transport-term}
       \begin{aligned}
            W_\infty(\rho_m,\tilde\rho_m)  \leq \Vert \Psi_{m} - \operatorname{Id} \Vert_{L^\infty(\R^d)} \lesssim |C_m - C_\infty|.
       \end{aligned}
    \end{equation}

    By \Cref{lem:W-infty-rate} 
    \begin{equation}\label{eq:fixed-support-term}
        W_\infty(\tilde \rho_m , \rho_\infty) \lesssim \frac{1}{m-1}.
    \end{equation}

    Therefore, plugging \eqref{eq:transport-term} and \eqref{eq:fixed-support-term} into \eqref{eq:w2-triangle} we have, for all sufficiently large $m$ and any $p\in [1,+\infty]$ 

    \begin{equation}
        W_p(\rho_m,\rho_\infty)\leq W_\infty(\rho_m,\rho_\infty) \lesssim \frac{1}{m-1} 
    \end{equation} 

It remains to show that the rate is sharp. Let $\varphi$ be a nonnegative, nonzero, $1$-Lipschitz function compactly supported in $\{0<C_\infty-V<1/4\}$. By \Cref{lem:cm-convergence}, for sufficiently large $m$ we have $(C_m-V(x))_+\leq1/2$ on $\spt(\varphi)$. Since $K_m\leq1$, the concavity of $z\mapsto z^{\frac{1}{m-1}}$ gives \[ \rho_m(x)\leq 1+\frac{(C_m-V(x))_+-1}{m-1}\leq1-\frac{1}{2(m-1)} \] on $\spt(\varphi)$. Since $\rho_\infty=1$ there, the Kantorovich--Rubinstein theorem gives \[ W_1(\rho_\infty,\rho_m)\geq\int_{\R^d}\varphi(x)(\rho_\infty-\rho_m)\d x\geq\frac{1}{2(m-1)}\int_{\R^d}\varphi(x)\d x\gtrsim\frac{1}{m-1}. \] As $\varphi$ is fixed, the implicit constant is independent of $m$. The lower bound for every $p\in[1,+\infty]$ follows by monotonicity of the Wasserstein distances, with a constant independent of $p$.

\end{proof}

\appendix
\section{}\label{app:geometric-results}

\begin{proof}[Proof of \Cref{prop:flow}]
  
Consider $\delta >0$ sufficiently small so that $C-3\delta >0$. Let $\chi$ be a smooth function such that $0\leq \chi \leq 1$ with $\chi$ equal to $1$ on $\{C-\delta \leq V \leq C+\delta\}$ and equal to $0$ outside $\{C-3\delta \leq V \leq C+3\delta\}$. Since $V$ is coercive, the set $\{ V \leq C+3\delta\}$ is compact, thus $\chi$ is compactly supported. Define the vector field \begin{equation}
  b(x) = \begin{cases} \chi(x) \frac{\nabla V(x)}{|\nabla V(x)|^2} \text{ for } x\neq 0 \\ 0 \text{ for } x= 0 \end{cases}
\end{equation} we have that $b$ is compactly supported as $\chi$ is compactly supported. 
Since $V$ is coercive, it admits a global minimizer. Any global minimizer is a critical point of $V$, and since $0$ is the unique critical point by assumption, it follows that $0$ is the unique minimizer of $V$. As $V(0)=0$, we have $V(x)>0$ for every $x\neq 0$. Moreover, by uniqueness of the critical point,
$\nabla V(x)\neq 0\text{ for every }x\neq 0.$
Since $\operatorname{spt}(\chi)\subset \{V\geq C-3\delta\}$
and $C-3\delta>0$, we have $0\notin \operatorname{spt}(\chi)$. Hence $\nabla V\neq 0$ on $\operatorname{spt}(\chi)$.
Since $V\in C^2(\R^d)$ and $\nabla V \neq 0$ on $\spt(\chi)$, we have $b \in C^1_c(\R^d ; \R^d)$.

Consider now the flow generated by $b$ as in \eqref{eq:flow}, by the standard existence and uniqueness results for autonomous ODEs (see \cite{Chi06} Ch. 1) it follows that $\Psi_s$ is a $C^1$-diffeomorphism for every $s \in \R$. Let us show the desired properties of $\Psi_s$: Let us first show that $\Psi_s(\{V\leq C\}) = \{V\leq C+s\}$.

Let $x\in \{V\leq C\}$, computing $\frac{d}{ds} V(\Psi_s(x))$ gives
\begin{equation}\label{eq:potential derivative}
    \begin{aligned}
        \frac{d}{ds} V(\Psi_s(x)) = \frac{\nabla V(\Psi_s(x)) \cdot \nabla V(\Psi_s(x)) }{|\nabla V(\Psi_s(x))|^2}\chi(\Psi_s(x)) = \chi(\Psi_s(x)) \leq 1,
    \end{aligned}
\end{equation}hence integrating \eqref{eq:potential derivative} on $[0,s]$ gives \begin{equation}
    V(\Psi_s(x)) \leq V(x) + s \leq C+s
\end{equation}thus $\Psi_s(x) \in \{V\leq C+s\}$ which shows that $\Psi_s(\{V \leq C\}) \subset \{V\leq C+s\}$. To show the other inclusion, let $y \in \{V \le C+s\}$ and assume, for the sake of contradiction, that $V(\Psi_{-s}(y)) > C$. Consider the backwards flows $t\mapsto \Psi_{-t}(y)$ for $t\in [0,s]$. A similar computation to \eqref{eq:potential derivative} gives
\begin{equation}
    \frac{d}{dt} V(\Psi_{-t}(y)) = - \chi(\Psi_{-t}(y)) \leq 0
\end{equation}thus, the map $t\mapsto V(\Psi_{-t}(y))$ is non-increasing. Thus, from the contradicting assumption and from $y \in \{V\leq C+s\}$ we have \[ C< V(\Psi_{-s}(y))\leq V(\Psi_{-t}(y)) \leq V(\Psi_{0}(y))= V(y) \leq C+s\] which means that the entire curve $t\mapsto \Psi_{-t}(y)$ lies in the region $\{C<V\leq C+\delta\}$ and thus $\frac{d}{dt} V(\Psi_{-t}(y)) = -1$. Integrating this equality on $[0,s]$ leads to
\begin{equation}
    V(\Psi_{-s}(y)) - V(y) = -s
\end{equation} and thus, from $y\in \{V\leq C+s\}$ we get $ V(\Psi_{-s}(y)) = V(y) - s \leq C$ which contradicts $V(\Psi_{-s}(y)) > C$. Hence $\Psi_{s}(\{V\leq C\}) = \{V\leq C+s\}$ holds for $0< s <\delta$.

For the case $-\delta < s < 0$, suffices to note that, from what was previously shown, $\Psi_{r}(\{V\leq C-r\}) = \{V\leq C\}$ for $0<r<\delta$, because $\chi = 1$ on the region $C-r\leq V < C $. From $\Psi^{-1}_{r} = \Psi_{-r}$ we obtain that $\{V\leq C-r\} = \Psi_{-r}(\{V\leq C\})$ and thus choosing $s= -r$ we obtain the result for $-\delta < s < 0$.

Let $0<s<\delta$, the integral form of \eqref{eq:flow} gives \begin{equation}
  \Psi_s (x) - x  = \int_0^s b(\Psi_\tau(x))\d \tau
\end{equation}since $b$ is continuous and compactly supported, $\Vert b \Vert_\infty < +\infty$ and thus
\begin{equation}\label{eq:flow-displacement}
  |\Psi_s - x| \leq \Vert b \Vert_\infty s 
\end{equation}
and for $-\delta < s <0$ the same inequality follows by observing that $\Psi_{-s} = \Psi^{-1}_s$, thus \begin{equation}
  \Vert \Psi_s - \operatorname{Id} \Vert_\infty \leq \Vert b \Vert_\infty|s| 
\end{equation}.

Differentiating \eqref{eq:flow}  we obtain
\begin{equation}\label{eq:flow-derivative}
    \frac{d}{ds} D \Psi_s(x) = D [b(\Psi_s(x))] = Db(\Psi_s(x)) D\Psi_s(x),
\end{equation} then, byJacobi's formula we have, since $D\Psi_s(x)$ is invertible, 
\begin{equation}\label{eq:jacobian-derivative}\begin{aligned}
    \frac{d}{ds}  \operatorname{det}(D\Psi_s(x)) &= \operatorname{det}(D\Psi_s(x)) \operatorname{tr}([D\Psi_s(x)]^{-1} \frac{d}{ds} D\Psi_s(x)) \\
    &=\operatorname{det}(D\Psi_s(x)) \operatorname{tr}([D\Psi_s(x)]^{-1}Db(\Psi_s(x)) D\Psi_s(x) ) \\
    &=\operatorname{det}(D\Psi_s(x)) \operatorname{tr}(Db(\Psi_s(x))  ) \\
    &= \operatorname{det}(D\Psi_s(x)) \nabla \cdot (b(\Psi_s(x)))
    \end{aligned}
\end{equation}thus, the function $(x,s) \mapsto \operatorname{det} (D\Psi_s(x))$ solves the following ODE
\begin{equation}
    \begin{cases}
         \frac{d}{ds}  \operatorname{det}(D\Psi_s(x)) = (\nabla \cdot b)(\Psi_s(x))\operatorname{det}(D\Psi_s(x)) \\
         \operatorname{det}(D\Psi_0(x))  = 1
    \end{cases}
\end{equation} therefore, we obtain
\begin{equation}\label{eq:jacobian-formula}
    \operatorname{det}(D\Psi_s(x)) = \exp\left(\int_0^s(\nabla \cdot b)(\Psi_\tau(x))\d \tau\right).
\end{equation} Using the standard estimate $|e^{a} - 1| \leq e^{|a|}|a|$ we obtain
\begin{equation}\begin{aligned}\label{eq:jacobian-estimate}
    \left | \operatorname{det}(D\Psi_s(x)) - 1 \right| &\leq  \exp\left(\int_0^s  |(\nabla \cdot b)(\Psi_\tau(x))| \d\tau \right)  \int_0^s| (\nabla \cdot b)(\Psi_\tau(x)) |\d \tau\\
    &\leq  \exp(\Vert \nabla \cdot b \Vert_\infty |s|) \Vert\nabla \cdot b\Vert_\infty \,  |s| \\
    &\leq  \exp(\Vert \nabla \cdot b \Vert_\infty \delta) \Vert\nabla \cdot b\Vert_\infty \,  |s| 
\end{aligned} 
\end{equation}

Thus, choosing $K\coloneq 2\max(\Vert b \Vert_\infty,  \exp(\Vert\nabla \cdot b \Vert_\infty \delta) \Vert \nabla \cdot b\Vert_\infty \,)$ and combining \eqref{eq:flow-displacement} and \eqref{eq:jacobian-estimate} we obtain \eqref{eq:flow-estimate}
\end{proof} 

\begin{proof}[Proof of \Cref{prop:area-formula}]

Let $C>0$, and let $b$, $\delta$ and $\Psi_s$ be as in the proof of \Cref{prop:flow}. Thus, we have $\Psi_s(\{V\leq C\}) = \{V\leq C+s\}$, hence $|\Psi_s(\{V\leq C\})| = A(C+s)$. By the change of variables we have \[A(C+s) = \int_{\{V\leq C\}} \det D \Psi_s(x) \d x.\]By \eqref{eq:jacobian-formula}, for $|s| < \delta$
\[A(C+s) = \int_{\{V\leq C\}}\underbrace{\exp\left( \int_0^s (\nabla \cdot b)(\Psi_\tau(x)) \d \tau\right)}_{g(x,s)} \d x.\] Since the function $(x,s) \mapsto g(x,s)$ is continuous, $(x,s)\mapsto \partial_sg(x, s)$ is continuous and bounded on $\{V\leq C\} \times [-\delta, \delta]$, we may differentiate under the integral sign to obtain \[\frac{d}{ds} A(C+s) = \int_{V\leq C} (\nabla \cdot b)(\Psi_s(x))\exp\left( \int_0^s (\nabla \cdot b)(\Psi_\tau(x)) \d \tau\right) \d x.\] Now we show that $s \mapsto \frac{d}{ds} A(C+s)$ is continuous: Let $s_n \to s$, then by continuity we have \[
\begin{aligned}
&(\nabla \cdot b)(\Psi_{s_n}(x))
\exp\left(
    \int_0^{s_n}
    (\nabla \cdot b)(\Psi_\tau(x))
    \d \tau
\right)
\\
&\qquad \longrightarrow
(\nabla \cdot b)(\Psi_s(x))
\exp\left(
    \int_0^s
    (\nabla \cdot b)(\Psi_\tau(x))
    \d \tau
\right).
\end{aligned}
\] and \[\left | (\nabla \cdot b)(\Psi_{s_n}(x))\exp\left( \int_0^{s_n} (\nabla \cdot b)(\Psi_\tau(x)) \d \tau\right) \right | \leq \Vert \nabla \cdot  b \Vert_\infty \exp(\Vert \nabla \cdot b \Vert_\infty  \delta)\]hence, continuity of $s\mapsto \frac{d}{ds} A(C+s)$ follows by dominated convergence. We thus obtain that $A$ is of class $C^1$ in a neighbourhood of $C$, and since $C>0$ was arbitrarily taken, we have $A\in C^1((0,+\infty))$.

We still must show that $A$ is continuous at $ 0$. From the assumptions on $V$ it follows that the unique minimizer of $V$ is $x=0$, thus $A(0) = |\{V\leq0\}| = 0$. Let $t_n$ be a decreasing sequence converging to zero, then \[\lim_{t_n \to 0} A(t_n) = \lim_{t_n \to 0} |\{V\leq t_n\}| = \left| \cap_{n} \{ V\leq t_n \} \right| = |\{V=0\}| = 0 = A(0).\]

By \eqref{eq:jacobian-derivative} we obtain \[\begin{aligned}
     A
'(C) &= \frac{d}{ds} A(C+s) \bigg \vert_{s=0} \\& = \int_{\{V\leq C\}} \frac{\d}{\d s} \det D\Psi_s(x) \bigg \vert_{s=0} \d x \\&= \int_{\{V\leq C\}} \nabla \cdot b \,  \d x,
\end{aligned} \] and using the divergence theorem, observering that the outward unit normal vector on the set $\{V=C\}$ is $\frac{\nabla V}{|\nabla V|}$, we have \[A'(C) = \int_{\{V=C\}} b \cdot \frac{\nabla V}{|\nabla V|} \d \mathcal H^{d-1} = \int_{\{V = C\}} \frac{1}{|\nabla V|}\d \mathcal H^{d-1}.\]

Finally, for every $t>0$ we have that $\{V=t\}\neq \emptyset$ and $\nabla V \neq 0$ on $\{V=t\}$, thus, \[A'(t) = \int_{\{V=t\}} \frac{1}{|\nabla V|} \d \mathcal{H}^{d-1} > 0,\] which implies that $A$ is strictly increasing on $(0,+\infty)$. Since $A(0) = 0 < A(t)$ for every $t>0$, $A$ is actually strictly increasing on the larger set $[0,+\infty)$.
\end{proof}

\section*{Acknowledgements}

This work was supported by the European Union via the ERC AdG 101054420 EYAWKAJKOS project.

\section*{AI statement}

No AI has been used in this work.

\bibliographystyle{abbrv}
\bibliography{references}

\end{document}